\documentclass[a4paper,11pt]{amsart}
\usepackage{dsfont}
\usepackage{amsmath,amsthm,amssymb,tabu,stmaryrd}
\usepackage{mathtools}
\usepackage[all]{xy}
\usepackage[usenames, dvipsnames]{color}
\usepackage{listings} 
\usepackage{graphicx}
\usepackage[table]{xcolor}
\usepackage{hyperref}
\hypersetup{
	colorlinks=true,
	linkcolor=RoyalBlue,
	filecolor=magenta,
	urlcolor=blue,
	citecolor=OliveGreen
}
\usepackage[noadjust]{cite}
\usepackage{tikz,tikz-cd}
\usepackage{enumitem}
\usetikzlibrary{shapes.geometric,calc,arrows}
\usepackage{multirow}

\usepackage[iso-8859-7]{inputenc}

\newtheorem{introductiontheorem}{Theorem}
\newtheorem{theorem}{Theorem}[section]
\newtheorem{corollary}[theorem]{Corollary}
\newtheorem{lemma}[theorem]{Lemma}
\newtheorem{proposition}[theorem]{Proposition}
\newtheorem{remark}[theorem]{Remark}
\newtheorem{definition}[theorem]{Definition}

\def\dim{\operatorname{dim}}

\def\id{\operatorname{id}}
\def\deg{\operatorname{deg}}

\def\Proj{\operatorname{Proj}}
\def\Spec{\operatorname{Spec}}

\def\BirMori{\operatorname{BirMori}}
\DeclareMathOperator{\BirMoriB}{\mathbf{BirMori}}
\def\Bir{\operatorname{Bir}}
\def\Aut{\operatorname{Aut}}

\def\PGL{\operatorname{PGL}}

\def\Cr{\operatorname{Cr}}

\newcommand{\OO}{\mathcal{O}}

\newcommand{\ZZ}{\mathbb{Z}}
\newcommand{\QQ}{\mathbb{Q}}
\newcommand{\RR}{\mathbb{R}}
\newcommand{\CC}{\mathbb{C}}
\newcommand{\PP}{\mathbb{P}}
\newcommand{\FF}{\mathbb{F}}
\renewcommand{\AA}{\mathbb{A}}
\newcommand{\kk}{\textbf{k}}

\newcommand{\isom}{\cong}
\newcommand{\defeq}{\vcentcolon=}

\newcommand{\map}{\smash{\xymatrix@C=0.5cm@M=1.5pt{ \ar[r]& }}}
\newcommand{\rmap}{\smash{\xymatrix@C=0.5cm@M=1.5pt{ \ar@{-->}[r]& }}}
\newcommand{\psmap}{\smash{\xymatrix@C=0.5cm@M=1.5pt{ \ar@{..>}[r]& }}}
\renewcommand{\d}[1]{\ensuremath{\operatorname{d}\!{#1}}}
\newcommand{\freeprod}{\mathop{\scalebox{1.7}{\raisebox{-0.1ex}{$\ast$}}}}

\definecolor{myorange}{HTML}{eb811a}

\title{Rigid birational involutions of $\PP^3$ and cubic threefolds}
\author{Sokratis Zikas}
\date{\today}
\address{Universit\"at Basel, Departement Mathematik und Informatik, Spiegelgasse 1, CH--4051 Basel, Switzerland}
\email{sokratis.zikas@unibas.ch}

\begin{document}

\maketitle

\begin{abstract}
	We construct families of birational involutions on $\PP^3$ or a smooth cubic threefold which do not fit into a non-trivial elementary relation of Sarkisov links. 
	As a consequence, we construct new homomorphisms from their group of birational transformations, effectively re-proving their non-simplicity. We also prove that these groups admit a free product structure.	
	Finally, we produce automorphisms of these groups that are not generated by inner and field automorphisms.
\end{abstract}

\section{Introduction}

\subsection{Homomorphisms from the Cremona group and free product structure}
The Cremona group $\Cr_n(\kk) = \Bir_{\kk}(\PP^n)$ is the group of birational transformations of the projective space $\PP^n$ over a field $\kk$. The study of this group has been a classical problem dating back to $19$th century.

The Cremona group in dimension $2$ over any field $\kk$ is known to be non-simple (see \cite{MR3037611,MR3533276}) i.e., it admits non-trivial homomorphisms to other groups. 
Recently, many families of such homomorphisms were constructed: for example in dimension $2$ over a perfect field by \cite{MR4153105, schneider2021relations} and over a subfield of the complex numbers by \cite{JEP_2020__7__1089_0}  in dimension $3$ and  by \cite{BLZ} in dimension greater or equal to $3$. 
Among other important consequences, the examples in the latter case proved for the first time the non-simplicity of the Cremona group in dimension greater or equal to $3$.

In this paper, we construct uncountable families of involutions of $\PP^3$ over $\CC$, which are Sarkisov links and do not fit into any non-trivial relations of Sarkisov links. 
These are links of Type II of the form
	\[
\xymatrix@R=.25cm{
	X \ar@{..>}[rr]^{\chi_C}\ar[dd] \ar[rd] && X\ar[dd] \ar[ld]\\
	& Z\\
	\PP^3 && \PP^3
}
\]
where $X \to \PP^3$ is a divisorial contraction to a curve $C$ and the central model $Z$ is a sextic double solid, whose covering map induces the involution. The link $\chi_C$ is completely determined by $C$ and the families of these links are parametrized by the Hilbert schemes of these curves.

Using these links we obtain the following result:

\begin{introductiontheorem}\label{main theorem 1}
	There exists a group homomorphism
	\[
	\psi\colon \Cr_3(\CC) \to \freeprod_I \ZZ/2\ZZ,
	\]
	where 
	\begin{enumerate}
		\item the indexing set $I$ parametrizes projective equivalence classes of certain curves and is uncountable (see section \ref{section:homoAndSDProd} for the precise description of $I$);
		\item $\PGL_4(\CC)$ lies in the kernel and
		\item there exist elements $\chi_i \in \Cr_3(\CC)$, $i\in I$, of degree $19$ not in the kernel.
	\end{enumerate}
	
	Moreover, $\psi$ admits a section giving the group $\Cr_3(\CC)$ a semi-direct product structure.
\end{introductiontheorem}

This can be thought of as a counterpart to the homomorphisms constructed in \cite{BLZ}. 
It should be noted that their results hold in all dimensions greater than or equal to $3$ and apply to many other classes of varieties;
however, the advantage of our construction lies in the fact that it is quite explicit,
thus proving the non-simplicity of $\Cr_3(\CC)$ in an effective way. 

Theorem \ref{main theorem 1} also provides the first example of a surjective group homomorphism $\Cr_3(\CC) \to \freeprod_I \ZZ/2\ZZ$, where we have specific examples of elements which are known to lie outside the kernel. This also contrasts the situation in dimension $2$ over $\CC$: 
in all proofs of non-simplicity of $\Cr_2(\CC)$, normal subgroups are constructed directly with the corresponding group homomorphism being the quotient one. These constructions are again non-effective, thus no elements of low degree are known to lie in the kernel.

Using a subset of the aforementioned involutions we obtain another structural result. More specifically, let $J$ be the subset of $I$ corresponding to curves which are fixed by no non-trivial automorphism of $\PP^3$, 
and denote by $G$ the subgroup of $\Cr_3(\CC)$ generated by all elements admitting a decomposition into Sarkisov links, none of them equivalent to $\chi_{C_j}$, $j\in J$ (see Remark \ref{equivalent links}) .
We then have the following:

\begin{introductiontheorem}\label{main theorem 2}
	The Cremona group $\Cr_3(\CC)$ can be written as the free product
	\[
	\Cr_3(\CC) = G \ast \left(\freeprod_J \langle \chi_{C_i} \rangle \right) \isom G \ast \left(\freeprod_J \ZZ/2\ZZ\right)
	\]
	where the indexing set $J$ is uncountable.
\end{introductiontheorem}

This is an analogue to \cite[Theorem C]{MR4153105}, where $\Cr_2(\kk)$ is shown to admit a similar free product structure when $\kk$ is a perfect field that admits a Galois extension of degree $8$.
In their construction the involutions $\chi_{C_i}$ are instead replaced by Bertini involutions.
The indexing set corresponds to points of degree $8$ in general position, not fixed by any automorphism of the plane and up to projective equivalence.

We will now briefly discuss the techniques used to produce the aforementioned constructions. The basic idea is to use the Sarkisov program. This is essentially an algorithm which decomposes any birational map between Mori fiber spaces into a sequence of simpler maps called Sarkisov links. The algorithm was proven to hold in dimension $2$, over perfect fields by \cite{MR1422227}, in dimension $3$, over $\CC$ by \cite{Corti} and in dimension greater than or equal to $3$, over $\CC$ by \cite{HM13}.

Using the Sarkisov program we get a set of generators, not quite for $\Cr_n(\kk)$, but for the groupoid $\BirMori_{\kk}(\PP^n)$. This is a groupoid whose objects are Mori fiber spaces birational to $\PP^n$ and whose morphisms are birational maps between them. Once we have a set of generators, we want to know the relations between them. This is made possible by the machinery of \emph{rank r fibrations} developed in \cite{BLZ} based on ideas from \cite{MR3123306}. This gives us a presentation of the groupoid $\BirMori_{\kk}(\PP^n)$, where relations are induced by rank $3$ fibrations. Once we have a presentation, we can construct groupoid homomorphisms to groups or groupoids and restrict them to get group homomorphisms from $\Cr_n(\kk)$.

\subsection{Non-generation of $\Aut(\Cr_3(\CC))$ by inner and field automorphisms}\label{subsection:intro nongeneration}
The group of field automorphisms of $\kk$ acts on $\PP^n_{\kk}$ naturally:
given $\tau \in \Aut(\kk)$ we may define the map $a_\tau$ as
\[
\begin{array}{ccc}
	\PP^n & \to & \PP^n\\
	(x_0:\ldots:x_n) & \mapsto & \big(\tau(x_0):\ldots:\tau(x_n)\big).
\end{array}
\]
Note that this is not a morphism defined over $\Spec(\kk)$. However, $\Aut(\kk)$ acts on the group $\Cr_n(\kk)$ by conjugation. Given a $\tau \in \Aut(\kk)$ we define a group automorphism $b_\tau$ as 
\[
\begin{array}{ccc}
	\Cr_n(\kk) & \to & \Cr_n(\kk)\\
	f & \mapsto & a_\tau \circ f \circ (a_\tau)^{-1}.
\end{array}
\] 
A quick calculation yields that if $f= (f_0:\ldots :f_n)$, where $f_i$ are homogeneous polynomials of the same degree having no common factor, then $b_{\tau}(f) = ({f_0}^{\tau}:\ldots:{f_n}^{\tau})$, where if $f_j = \sum a_Ix^I$, then ${f_j}^{\tau} = \sum \tau(a_I)x^I$. 

In \cite{MR2278755}, the group $\Aut(\Cr_2(\CC))$ was shown to be generated by inner and field automorphisms, that is if $\phi\colon \Cr_2(\CC) \to \Cr_2(\CC)$ is a group homomorphism, then there exists a field automorphism $\tau$ of $\CC$ and an element $g\in \Cr_2(\CC)$ such that for every $f \in \Cr_2(\CC)$ we have
\[
\phi(f) = g\circ b_{\tau}(f) \circ g^{-1}.
\]
It is therefore a natural question to ask whether such a result is true in higher dimensions or over other fields. 
In this text, we give a negative answer in dimension $3$ over $\CC$:

\begin{introductiontheorem}\label{main theorem 3}
	There exist uncountably many automorphisms of $\Cr_3(\CC)$ of arbitrary order which are not generated by inner and field automorphisms.
\end{introductiontheorem}

These automorphisms are constructed using the free product structure on $\Cr_3(\CC)$ of Theorem \ref{main theorem 2}. They act on the generators by exchanging two elements of the form $\chi_{C_j}$ and $\chi_{C_{j'}}$.
The fact that such an automorphism is not inner boils down to the fact that these involutions do not fit into a non-trivial relation of Sarkisov links, while a correct choice of $C_j$ and $C_{j'}$ shows that the automorphism is not a field automorphism up to inner ones.

Finally, in \cite{ContinuousAutomorphisms}, the authors prove that any homeomorphism of $\Cr_3(\kk)$, with respect to either the Zariski or the Euclidean topology, is a composition of an inner and a field automorphism for $\kk = \RR$ or $\CC$. 
Thus our examples constitutes, to our knowledge, the first examples of non-continuous automorphisms of $\Cr_3(\CC)$.

\subsection{Extensions of our results to cubic $3$-folds}

All three of our theorems extend to the case of the group of birational automorphisms of a smooth cubic $3$-fold $Y$. 

For Theorems \ref{main theorem 1} and \ref{main theorem 2}, the same construction applies to any smooth cubic $3$-fold unconditionally. 
In the case of Theorem \ref{main theorem 1}, we note again that the results of \cite{BLZ} still apply to the case of $\Bir_{\CC}(Y)$. Again, the advantage of our result lies in its explicit nature.
For instance, our approach provides examples of elements of order as low as $11$ not in the kernel of the homomorphism $\Bir_{\CC}(Y) \to \freeprod_I \ZZ/2\ZZ$.

Finally, for Theorem \ref{main theorem 3} the action of a field automorphism $\tau$ on $\Aut(\Bir_{\CC}(Y))$ is well defined if and only if $\tau$ preserves $Y$, that is $a_{\tau}(Y) = Y$. Thus the statement of the corresponding theorem must be modified accordingly.

\subsection*{Acknowledgment}I would foremost like to thank \emph{J\'er\'emy Blanc} for suggesting the problem, as well as for his guidance and help throughout. I would also like to thank \emph{St\'ephane Lamy} and \emph{Christian Urech} for suggesting the application to the automorphism group of $\Cr_3(\CC)$, and \emph{Serge Cantat}, \emph{Erik Paemurru},  \emph{Nikolaos Tsakanikas} and \emph{Immanuel van Santen} for the interesting discussions and remarks.
Finally, I would like to thank the anonymous referees for their comments.

This work was supported by the Swiss National Science Foundation Grant ``Birational transformations of threefolds'' $200020\_178807$.

\section{Preliminaries}

In the rest of the paper all varieties and birational maps between them are defined over~$\CC$.

\subsection{Rank $r$ fibrations and elementary relations}

Here, we give a brief account of the theory developed in \cite[Sections 3 and 4]{BLZ}. Any proofs provided here are sketches of the actual proofs found there.

\begin{definition}
	Let $X/B$ be a Mori fiber space with singularities not worse that terminal (terminal Mori fiber space for short). We define $\BirMoriB(X)$ to be the groupoid whose objects are terminal Mori fiber spaces, birational to $X$ and the morphisms between them to be birational maps.
\end{definition}

\begin{definition}
	Let $X/B$ and $X'/B'$ be Mori fiber spaces. An isomorphism between $X$ and $X'$ is called an \textbf{isomorphism of Mori fiber spaces} if there exists an isomorphism between $B$ and $B'$ that makes the induced diagram commute.
\end{definition}

\begin{definition}\label{rank r fibration}
	Let $r \geq 1$ be an integer. A morphism $\eta \colon X \map B$ is a \textbf{rank $\boldsymbol{r}$ fibration} if the following conditions hold:
	\begin{enumerate}
		\item the fiber space $X/B$ given by $\eta$ is a relative Mori Dream Space (see \cite[Definition 2.2]{BLZ});
		\item $\dim X > \dim B$ and $\rho(X/B) = r$;
		\item $X$ is $\QQ$-factorial and terminal and for any divisor $D$ on $X$, the output of
		any $D$-MMP over $B$ is still $\QQ$-factorial and terminal.
		\item There exists an effective $\QQ$-divisor $\Delta_B$ such that the pair $(B,\Delta_B)$ is klt.
		\item The anticanonical divisor of $X$ is $\eta$-big.
	\end{enumerate}

	We say that a rank $r$ fibration $X/B$ \textbf{dominates} a rank $r'$ fibration $X'/B'$ if we have a commutative diagram
	\[
	\xymatrix@R=0.1cm{
	X \ar[rrr] \ar@{-->}[dr]&&& B\\
	& X' \ar[r] & B' \ar[ur]
	}
	\]
	where $X \rmap X'$ is a birational contraction and $B' \map B$ is a morphism with connected fibres.
\end{definition}

\begin{remark}\label{rank 1 and 2 fibrations}
	A rank $1$ fibration $\eta\colon X \map B$ is a terminal Mori fibre space. Indeed, the only thing left to check is the relative ampleness of the anti-canonical divisor. However, since $-K_X$ is $\eta$-big, we may write
	\[
	-K_X \equiv A + E,
	\]
	where $A$ is $\eta$-ample and $E$ is effective. Since $\rho(X/B) = 1$, $E$ is either $\eta$-nef or $\eta$-anti-nef. Since the contracted curves cover $X$, an effective divisor cannot be $\eta$-anti-nef, thus $E$ is $\eta$-nef and subsequently, $-K_X$ is $\eta$-ample.
	
	Similarly, rank $2$ fibrations correspond to Sarkisov links between two Mori fibre spaces in the following manner:
	
	If $X/B$ is a rank $2$ fibration then we may run a $(-A)$-MMP over $B$ for any ample divisor $A$. Then since $\rho(X/B) = 2$, at the first step we have a choice between 2 rays to contract giving us 2 different MMPs. Since $\kappa(-A) = -\infty$, the output of both MMPs must be rank $1$ fibrations, which correspond to Mori fibre spaces.
	
	On the other hand, let
	
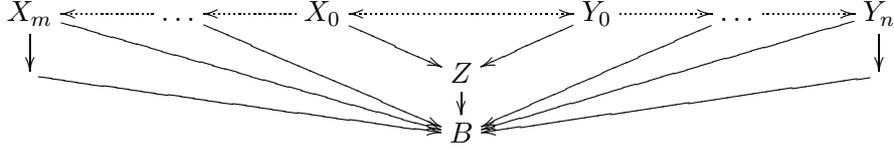
\[
	\xymatrix@C=1.2cm@R=0.3cm{
		X_m\ar[d] \ar[rrrdd] & \dots \ar@{..>}[l] \ar[rrdd] & X_0 \ar[rd] \ar@{..>}[l] \ar@{<..>}[rr] & & Y_0 \ar[ld] \ar@{..>}[r] &\dots \ar@{..>}[r] \ar[lldd]&Y_n \ar[d] \ar[llldd]\\
		\ar[rrrd]&&& Z \ar[d] &&& \ar[llld]\\
		&&& B &&&
	}
	\]
	be a Sarkisov diagram, where $X_0 \psmap Y_0$ is either a flop or an isomorphism. Then $X_0/B$ is weak Fano thus a Mori Dream Space. Moreover $X_0$ is $\QQ$-factorial and terminal and the output of any MMP is among the maximal dimensional varieties appearing in  the diagram, which by assumption are all $\QQ$-factorial and terminal. Finally, the fact that $B$ is klt is proven in \cite[Corollary 4.6]{Fujino}.
	
	The correspondence above is not \emph{one-to-one}, namely a rank $2$ fibration gives rise to a Sarkisov link and its inverse, up to Mori fiber space isomorphisms. On the other hand, in the Sarkisov diagram above, all $X_i/B$ and $Y_i/B$ are rank $2$ fibrations.
\end{remark}

\begin{proposition}[{\cite[Proposition 4.3]{BLZ}}]\label{rank 3 fibrations correspond to elementrary relations}
	Let $X \map B$ be a rank $3$ fibration. Then there are only finitely many rank $2$ fibrations, corresponding to Sarkisov links $\chi_i$ (up to Mori fiber space isomorphisms), dominated by $X/B$, and they fit in a relation 
	\[
	\chi_t\circ \dots \circ \chi_1 = id.
	\]
	
\end{proposition}

\begin{definition}
	A \textbf{trivial relation} between Sarkisov links is a relation of one of the following forms
	\[
	\phi^{-1} = \psi \quad \text{ and } \quad \alpha \circ \phi \circ \beta = \psi,
	\]
	where $\phi, \psi$ are Sarkisov links and $\alpha, \beta$ are isomorphisms of Mori fiber spaces.
	
	An \textbf{elementary relation} between Sarkisov links  is one that arises from a rank $3$ fibration (see Proposition \ref{rank 3 fibrations correspond to elementrary relations}).
\end{definition}

\begin{theorem}[{\cite[Theorem 1.1]{HM13}, \cite[Theorem 4.28]{BLZ}}]\label{elementary relations are generators}
	Let $X/B$ be a terminal Mori fibre space.
	\begin{enumerate}
		\item The groupoid $\BirMori(X)$ is generated by Sarkisov links and isomorphisms of Mori fiber spaces.
		\item Any relation between Sarkisov links in $\BirMori(X)$ is generated by trivial and elementary relations.
	\end{enumerate}
\end{theorem}

\begin{remark}\label{equivalent links}
	The first part of the theorem is due to \cite{HM13}. The original version does not mention the isomorphisms of Mori fiber spaces, which are however implicit in their proof.
	Note that an isomorphism between the total spaces of two Mori fiber spaces which is not a Mori fiber space isomorphism is a non-trivial Sarkisov link.
	
	Similarly, the second part of the original theorem in \cite{BLZ} does not mention the trivial relations as generators. These are indeed ``trivial'' from a birational point of view. For our purposes though, we will need a slightly more accurate statement and so we explain the subtleties.
	
	For the first type of trivial relation, a rank $2$ fibration corresponds to a unique Sarkisov diagram up to composition with Mori fiber space isomorphisms on the left and right. However, as already discussed in Remark \ref{rank 1 and 2 fibrations}, a Sarkisov diagram is not directed and thus corresponds to both a link and its inverse.
	The second type of relation is just a by-product of not working up to Mori fiber space isomorphism. 
	
	Note that a trivial relation involving a link $\chi$ arises from the rank $2$ fibration that corresponds to $\chi$ up to orientation and compositions with isomorphisms.
	Thus, by definition, an elementary relation cannot be a trivial one.

	With that in mind, we will say that two Sarkisov links $\phi$ and $\psi$ are \textbf{equivalent} if there exist $\alpha, \beta$, isomorphisms of Mori fiber spaces such that $\alpha \circ \phi \circ \beta = \psi$.
\end{remark}

\subsection{Weighted blowups}

\begin{definition}\label{weighted blowup}
	Let $\mathbf{w} = (w_1,\dots,w_n)$ be positive integers. Define the $\CC^*$-action on $\AA^{n+1}$ by
	\[
	\lambda\cdot (u,x_1,\dots,x_n) = (\lambda^{-1}u,\lambda^{w_1}x_1,\dots,\lambda^{w_n}x_n).
	\]
	The morphism from the geometric quotient $T \defeq \AA^{n+1}/\CC^* \to \AA^n$ defined by
	\[
	\begin{array}{ccc}
		T & \map & \AA^n\\
		(u:x_1:\dots:x_n) & \mapsto & (u^{w_1}x_1,\dots, u^{w_n}x_n)
	\end{array}
	\]
	is called the \textbf{standard} $\mathbf{w}$\textbf{-blowup} of $\AA^n$ at the origin.
	
	Let $f\colon E \subset X \to p \in Y$ be a morphism contracting a divisor $E$ to a smooth point $p$. We say that $f$ is a $\mathbf{w}$\textbf{-blowup} of $Y$ at $p$ if there exists an analytic neighbourhood $(U,p) \isom (\AA^n,0)$ of $p$ such that  the restriction $f|_{f^{-1}(U)}\colon f^{-1}(U) \to U$ is the standard $\mathbf{w}$-blowup of $\AA^n$ at $0$.
\end{definition}

\begin{lemma}\label{Discrepancy of (1,a,b)-blowup}
	Let $p \in Y$ be a smooth point of a $3$-fold and let $\pi\colon (X,E) \map (Y,p)$ be a $(1,a,b)$-blowup of $Y$ at $p$. Then the ramification formula takes the form
	\[
	K_X = \pi^*K_Y + (a+b)E.
	\]
\end{lemma}

\begin{proof}
	Since this is something that can be checked locally, up to local analytic isomorphism we may assume that $(Y,p) = (\AA^3,0)$, $X$ is the quotient $\AA^4/\CC^*$ under the action 
	\[
	\lambda\cdot (u,x_1,x_2,x_3) = (\lambda^{-1}u,\lambda x_1, \lambda^a x_2, \lambda^b x_3)
	\]
	and $\pi$ is given by $(u:x_1:x_2:x_3) \mapsto  (ux_1,u^ax_2,u^bx_3)$.
	
	Let $U_1$ be the open subset $\{x_1 \neq 0\} \subset X$, isomorphic to $\AA^3$. If we denote the composition
	\[
	\begin{array}{ccccc}
		\AA^3 & \map & U_1 \subset X & \map & \AA^3\\
		(v,y_1,y_2) & \mapsto & \left(v:1:y_1:y_2\right) & \mapsto & (v,y_1v^a,y_2v^b).
	\end{array}
	\]
	by $\psi$ then we may calculate that
	\[
	\psi^*(1\d x_1\wedge \d x_2 \wedge \d x_3) = v^{a+b}\d v\wedge \d y_1 \wedge \d y_2.
	\]
	Taking the divisor of this $3$-form we conclude.
\end{proof}

\begin{lemma}\label{intersection with exceptional}
	Let $p \in Y$ be a smooth point of a $3$-fold and let $\pi \colon E\subset X \map p \in Y$ be a $(1,a,b)$-blowup of $Y$ at $p$. Let $\Gamma$ be a curve in $Y$ which is a complete intersection in an affine neighbourhood $U$ of $p$. Choose generators $f_1$ and $f_2$ for the ideal of regular functions on $U$ vanishing along $\Gamma$. We then have
	\[
	E \cdot \tilde{\Gamma} = \frac{v_E(f_1)\cdot v_E(f_2)}{a b},
	\]
	where $\tilde{\Gamma}$ denotes the strict transform of $\Gamma$, $v_E$ is the divisorial valuation defined by $E$ and $f_1$ and $f_2$ are considered as rational functions on $X$.
\end{lemma}

\begin{proof}
	Again we will work in a local analytic neighbourhood and assume that $(Y,p) = (\AA^3,0)$ and $\pi$ is given by $(u:x_1:x_2:x_3) \mapsto  (ux_1,u^ax_2,u^bx_3)$.
	We may write
	\[
	f_n = \sum_{i=k_n}^{d_n}h_{n,i}(v,y_1,y_2),
	\]
	for $n=1,2$, where $h_{n,i}$ are homogeneous polynomials with respect to the grading $(1,a,b)$ and $h_{n,k_n} \neq 0$. 
	Then, pulling back under $\pi$ we get
	\[
	\pi^*(f_n)  = f_n(ux_1,u^a x_2,u^b x_3) = u^{k_n}\left( \sum_{i=k_n}^{d_n}u^{i-k_n} h_{n,i}(x_1,x_2,x_3)\right),
	\]
	which shows that $v_E(f_n) = k_n$. 
	Moreover the ideal of $\tilde{\Gamma}$ is generated by $\tilde{f_1}$ and $\tilde{f_2}$ with
	\[
	\tilde{f_n} = \sum_{i=k_n}^{d_n} u^{i-k_n} h_{n,i}(x_1,x_2,x_3), 
	\]
	 for $n=1,2$.
	 Finally, using the fact that $E \isom \PP(1,a,b)$ and that $\tilde{\Gamma}$ is given by the vanishing of the $\tilde{f_n}$, $n=1,2$, we may compute that 
	 \[
	 E\cdot \tilde{\Gamma} = \mathbb{V}(\tilde{f_1})|_E \cdot \mathbb{V}(\tilde{f_2})|_E = \mathbb{V}(h_{1,k_1}) \cdot_{E} \mathbb{V}(h_{2,k_2}) = \frac{k_1\cdot k_2}{ab}.
	 \]
\end{proof}

\section{The construction}

Throughout this section, $Y$ will denote either $\PP^3$ or a smooth cubic $3$-fold in $\PP^4$.
Denote by $\mathcal{H}_{g,d}^Y$ the Hilbert scheme of subvarieties of $Y$ with Hilbert polynomial $P(n) = dn - g +1$.

\begin{proposition}\label{generalities on the involutions}
	Consider the following pairs $(g,d)$ depending on $Y$:
	\[\arraycolsep=6pt\def\arraystretch{1.2}
	\begin{array}{|c|c|}
		\hline
		\rowcolor{gray!50}
		Y & (g,d)\\
		\hline
		\PP^3 & (2,8), (6,9), (10,10), (14,11)\\
		\hline
		\rowcolor{gray!20}
		\text{Cubic $3$-fold} & (0,5), (2,6)\\
		\hline
	\end{array}
	\]
	Then there exists an irreducible component $\mathcal{S}_{g,d}^Y$ of  $\mathcal{H}_{g,d}^Y$ whose general element $C \in \mathcal{S}_{g,d}^Y$ is a smooth curve satisfying the following:
	if $X \map Y$ is the blowup of $Y$ along $C$ then:
	\begin{enumerate}
		\item $X$ is a smooth weak-Fano $3$-fold, there are finitely many $(-K_X)$-trivial curves and $\lvert -K_X \rvert$ is base-point free;
		\item The anti-canonical model $Z \defeq \Proj\left(\oplus_{n\geq 0}H^0(X,-nK_X)\right)$ of $X$ is a sextic double solid, that is a double cover of $\PP^3$ ramified along a sextic hypersurface;
	\end{enumerate}
	(see Lemma \ref{I is uncountable} for an estimation of the dimension of $\mathcal{S}_{g,d}^Y$).
\end{proposition}

\begin{proof}
	For the non-emptiness of $\mathcal{S}_{g,d}^Y$ we refer to \cite[Section 5.1]{WeakFanos} and \cite[Section 3.3]{WeakFanosCubics} for the cases of $\PP^3$ and a smooth cubic $3$-fold respectively.
	
	Similarly, the proof of $(1)$ can be found in \cite[Proposition 5.11]{WeakFanos} and \cite[Proposition 3.7]{WeakFanosCubics} for the two cases respectively.
	
	As for $(2)$,	
	we first note that in all cases, using the formula
	\[
	(-K_X)^3 = (-K_Y)^3 + 2K_Y\cdot C + 2g - 2
	\]
	we get $(-K_X)^3 = 2$.
	By the Hirzebruch-Riemann-Roch theorem (see \cite[pg. 437, Ex. 6.7]{Hartshorne}) together with the Kawamata-Viehweg vanishing theorem we get
	\[
	h^0(X,-nK_X) = \frac{n(n+1)(2n+1)}{12}(-K_X^3) + 2n+1 = \frac{n(n+1)(2n+1)}{6}+ 2n+1.
	\]	
	For $n = 1$ we get $h^0(X,-K_X) = 4$; we write $x_0, x_1, x_2, x_3$ for the generators. 	
	By $(1)$ the linear system $|-K_X|$ is base-point free and the associated morphism $X \to \PP(H^0(X,-K_X))$ contracts finitely many curves and is thus dominant. In particular, the $x_i$'s satisfy no polynomial relation. Moreover, since a general  element of $|-K_X|$ is the pullback of a general hyperplane and $(-K_X)^3=2$, the projection formula (see \cite[1.9]{Debarre}) implies that $X \to \PP(H^0(X,-K_X))$ is generically $2$ to $1$.
	For $n = 2$ we get $h^0(X,-2K_X) = 10 = \dim S^2H^0(X,-K_X)$. Since there is no relation between the $x_i$'s, we get the equality of these two spaces.
	For $n = 3$ we get $h^0(X,-3K_X) = 15 = \dim S^3H^0(X,-K_X) + 1$. Again using the fact that there is no relation between the $x_i$'s, we get that we only have one new generator. That is
	\[
	H^0(X,-3K_X) = S^3H^0(X,-K_X) \oplus  \langle t \rangle.
	\]
	We now consider the diagram
	\[
	\xymatrix@R=.4cm{
	X  \ar[rr] \ar[rd] && X' \ar[ld]\\
			& \PP^3
	}
	\]
	with $X' = \Proj(R)$, where $R$ is the graded algebra generated by $x_0,\dots,x_3$ with degrees $1$ and $t$ with degree $3$ and $X' \to \PP^3$ is the projection to the first four factors. Note that $X \to X'$ and $X \to \PP^3$ both contract the $(-K_X)$-trivial curves. Moreover, if $X' \to \PP^3$ were generically one to one, it would be a bijection and thus an isomorphism from Zariski's Main Theorem. Thus $X' \to \PP^3$ is two to one, which implies that $X \to X'$ has connected fibers. 
	For $n\geq 4$, the morphism given by $|-nK_X|$ contracts the same curves as $X \to X'$ and has connected fibers. Thus by \cite[Proposition 1.14]{Debarre}, these two morphisms are the same up to isomorphism. In particular, there is no new generator for any $n \geq 4$. Finally, since we know that the algebra $\oplus_{n\geq 0}H^0(X,-nK_X)$ is generated by $x_0,\dots,x_3,t$, 
	we only have to calculate the dimensions of the graded components to see that we have only one relation in degree $6$:
	indeed, for $n = 4,5$ we have
	\[
	h^0(X,-nK_X) = \dim \CC[x_0,\dots,x_3,t]_n;
	\]
	however, for $n = 6$ we have
	\[
	h^0(X,-6K_X) = 104 = \dim \CC[x_0,\dots,x_3,t]_n - 1,
	\]
	which shows that there is a relation in degree $6$ which, up to change of coordinates, can be brought to the form $F(x_0,\dots,x_3,t) = t^2 - f_6(x_0,\dots,x_3) = 0$;
	for any $n \geq 7$ we have
	\[
	h^0(X,-nK_X) = \dim \left(\frac{\CC[x_0,\dots,x_3,t]}{(F(x_0,\dots,x_3,t))}\right)_n.
	\]

\end{proof}

\begin{remark}\label{birational self-map}
	The construction above induces a birational self-map of $Y$ in the following way: denote by $\eta$ the rational map $Y \rmap X \map Z$ and by $p$ the deck transformation of $Z$ over $\PP^3$. Then $\chi_C \defeq \eta^{-1}\circ p \circ \eta \colon Y \rmap Y$ defines a birational map. Note that $\chi_C$ is an involution. Schematically, we have the diagram
	\[
	\xymatrix@R=.3cm{
	X \ar[dd] \ar[rd] \ar@{..>}[rr] && X \ar[ld] \ar[dd]\\
						& Z \ar@(dl,dr)_{p}\\
	Y && Y
	}
	\]
\end{remark}

\begin{remark}
	In the setting of Proposition \ref{generalities on the involutions}, any curve $\gamma$ contracted by $X \to Z$ is smooth and rational with normal bundle isomorphic to $\OO_{\PP^1}(a)\oplus  \OO_{\PP^1}(b)$, with $(a,b) = (-1,-1)$ or $(0,-2)$.
	
	Indeed, consider the contraction
	\[
	\gamma \subset S \subset X \map p \in H \subset Z
	\]
	where $p$ is a singular point of $Z$, $H$ is a general hyperplane section though $p$ and $S$ is the strict transform of $H$.
	Since $H$ is general, we may assume that its only singularity is $p$.
	By Theorem \ref{generalities on the involutions} $S$ is smooth and thus the morphism $S \to H$ factors through the minimal resolution $E \subset T \map p \in H$ of $H$, where $E$ is a chain of smooth rational curves.
	However, since the relative Picard rank of $S \to H$ is $1$, the morphism $S \to T$ must be an isomorphism.
	In particular, $\gamma$ is a smooth rational curve.

	Now denote by $E\subset W \to \gamma \subset X$ the blowup of $X$ along $\gamma$.
	Then $E$ is isomorphic to the Hirzebruch surface $\FF_{a - b}$. 
	Then by \cite[Lemma 2.2.14]{AGV} we have
	\[
	a+b = \deg\left(N_{\gamma/X}\right) = (-K_X)\cdot \gamma + 2g(\gamma) - 2 = -2.
	\]
	Moreover, using adjunction formula as well as the formulas in \cite[Lemma 5.5]{zikas2020sarkisov} we may compute
	\[
	-K_W|_E = -K_E + E|_E = -\frac{1}{2}K_E.
	\]
	Since $-K_W$ is nef (see case 3 in the proof of Proposition \ref{bigness}) so is $K_E$.
	Thus $E \isom \FF_n$ with $n = 0, 1$ or $2$, that is $a - b = 0, 1$ or $2$.
	The only integer solutions to the two equations are $(a, b) = (-1,-1)$ and $(0,-2)$. 
\end{remark}

\begin{proposition}\label{degree}
	Let $Y$, $C$ and $X$ be as above and let $H$ denote a hyperplane if $Y$ is $\PP^3$ and a hyperplane section otherwise. Then the degree of $\chi_C$ with respect to $H$ is 
	\[
	\deg(\chi_C) = (r^2H^3-d)r -1,
	\]
	where $d$ is the degree of $C$ and $r$ is the index of $Y$.
\end{proposition}

\begin{proof}
	We consider the induced diagram
	\[
	\xymatrix@R=.25cm{
	X \ar@{..>}[rr]^{\phi} \ar[dd] \ar[rd] && X\ar[dd] \ar[ld]\\
	& Z\\
	Y && Y
	}
	\]
	where $\phi$ is a flop over $Z$. Fix the basis $(K_X,H)$ for the $\QQ$-vector space $N^1(X)$, where, by abuse of notation, we denote again by $H$ the class of the pullback $H$. Then $\phi$ induces an automorphism of $N^1(X)$, by pullback, and since $K_X$ is an eigenvector for it, the associated matrix has the form 
	\[\phi^* = 
	\begin{pmatrix}
		1 & a\\
		0 & b
	\end{pmatrix}.
	\]
	Since $\phi^2 = id_X$, $b=-1$. Thus $\phi^*H = aK_X - H$.
	
	Using the formulas in \cite[Lemma 2.2.4]{AGV}, we may compute that
	\[
	(K_X)^2 \cdot H = r^2 H^3 - d,
	\]
	where $r$ is the index of $Y$ and $d$ the degree of $C$, and
	\[
	(\phi^*K_X)^2\cdot \phi^*H = (K_X)^2 \cdot (aK_X - H) = a(K_X)^3 - K_X\cdot H = -2a - (r^2H^3 -d).
	\]
	Equating the above formulas we get $a = -(r^2H^3 -d)$. Thus 
	\[
	\phi^*H = -(r^2H^3 -d)K_X - H = ((r^2H^3 -d)r- 1)H - (r^2 -d)E
	\]
	from which we conclude that ${\chi_C}^*(H) =((r^2H^3 -d)r - 1)H$.
\end{proof}

For the pairs of genus and degree of Proposition \ref{generalities on the involutions}, we obtain the following values for the degree of $\chi_C$:
\[\arraycolsep=6pt\def\arraystretch{1.2}
\begin{array}{|c|c|c|c|c|c|}
	\hline  
	 \cellcolor{gray!50}& \cellcolor{gray!20} (g,d) 				& \cellcolor{gray!20}(2,8) & \cellcolor{gray!20} (6,9) & \cellcolor{gray!20} (10,10) & \cellcolor{gray!20} (14,11)\\
	\cline{2-6}	
	\multirow{-2}{*}{\cellcolor{gray!50}$\PP^3$}& \deg(\chi_C)   & 31     &27      &23           & 19\\
	\hline
	\cellcolor{gray!50}&  \cellcolor{gray!20}(g,d) 				& \multicolumn{2}{c|}{\cellcolor{gray!20}$(0,5)$} & \multicolumn{2}{c|}{\cellcolor{gray!20}$(2,6)$} \\
	\cline{2-6}	
	\multirow{-2}{*}{\cellcolor{gray!50} \shortstack{Cubic \\$3$-fold}} & \deg(\chi_C)   & \multicolumn{2}{c|}{$13$}      &\multicolumn{2}{c|}{$11$}\\
	\hline
\end{array}
\]

\begin{proposition}\label{bigness}
	Let $X$ be as in Proposition \ref{generalities on the involutions} and $\pi\colon (W,E) \map (X,z)$ be a divisorial contraction with $W$ $\QQ$-factorial and terminal. Then $-K_W$ is not big.
\end{proposition}

\begin{proof}
	We first note that since $W$ is terminal and $X$ is smooth, by \cite[Proposition 1.2]{ContractionsToCurves} and \cite[Theorem 1.2]{Kawakita}, $W \map X$ is either the regular blowup of a curve or a $(1,a,b)$-blowup of a point, with $a,b$ coprime.
	
	We distinguish 3 cases based on the geometry of the center $z$: \\
	\textbf{Case 1:} $z$ is a point and $W \map X$ is a $(1,a,b)$-blowup.
	
	Suppose for contradiction that $-K_W$ is big and let $S_W \in \lvert -nK_W \rvert$ be a general element, for $n \gg 1$. Denote by $S_X$ the image of $S_W$ in $X$, by $H_X$ the pullback of a general hyperplane section $H_Z$ of $Z$ containing the image of $z$ and $\Gamma \subset X$ the intersection of $S_X$ with $H_X$. First notice that $S_X \in  \lvert -nK_X \rvert$ and $H_X \in  \lvert -K_X \rvert$. We thus have
	\[
	(-K_X)\cdot \Gamma = n(-K_X)^3 = 2n. 
	\]
	If we denote by $\Gamma_W$ the strict transform of $\Gamma$ in $W$, then by Lemma \ref{intersection with exceptional} we have 
	\[
	E\cdot \Gamma_W = \frac{v_E(S_X)\cdot v_E(H_X)}{ab} = \frac{n(a+b)}{ab}v_E(H_X),
	\]
	where the second equality follows from the ramification formula of Lemma \ref{Discrepancy of (1,a,b)-blowup}. Again, using the same formula we may compute that
	\[
	(-K_W)\cdot \Gamma_W = n\left(2 - \frac{(a+b)^2}{ab}v_E(H_X)  \right).
	\]
	Since we chose $H_X$ to be the pullback of a hyperplane containing the image of $z$, $v_E(H_X) \geq 1$. 
	The quantity $\frac{(a+b)^2}{ab}$ is always strictly greater than $2$, and so $(-K_W)\cdot \Gamma_W<0$.
	Finally, since we assumed that $-K_W$ is big then the sections of $-nK_W$ cover $W$ for sufficiently large $n$ and so do the curves $\Gamma$ chosen as above. This gives us a dense subset of $W$ covered by $(-K_W)$-negative curves, which contradicts the bigness of $-K_W$.
	\\	
	\textbf{Case 2:} $z$ is a curve not contracted by $X \map Z$.
	
	We have 
	\[
	-nK_W = \pi^*(-nK_X) - nE.
	\]
	Sections of $-nK_W$ are pullbacks of degree $n$ hypersurface sections of $Z$ vanishing along the curve $C \defeq \pi(z)$ with multiplicity $n$. Let ${h=0}$ be such a hypersurface section and $I = (f_1,\dots,f_k)$ be the ideal of $C$. Then $h \in I^n$ and since $\deg(h) = n$, $h$ can only be a linear combination of degree $n$ monomials in the linear elements in $I$. Thus for $-nK_W$ to be big, we need to have at least $4$ linear elements in $I$ which is a contradiction.\\	
	\textbf{Case 3:} $z$ is a curve contracted by $X \map Z$.
	
	In this case we consider the diagram 
	\[
	\xymatrix{
		W \ar[d]_g \\
		X \ar[rd]_f & F\ar[d]^r\\
		& \PP^3
	}
	\]
	where $f\colon X \to \PP^3$ is the morphism given by $|-K_X|$ and $r\colon F \to \PP^3$ is the blowup of the image $p \in \PP^3$ of $\gamma$ under $f$.
	Since the preimage of $p$ under $f \circ g$ is a Cartier divisor, $f\circ g$ factors through $r$ via $s\colon W \to F$.
	Finally, sections of $-K_W$ are pullbacks of hyperplanes of $\PP^3$ through $p$.
	Thus the previous diagram completes to the following:
	\[
	\xymatrix{
		W \ar[d]_g \ar[rd]^s \ar@/^{0.7cm}/[rrdd]^{\lvert -K_W \rvert}\\
		X \ar[rd]_f & F\ar[d]^r \ar[rd]\\
		& \PP^3 \ar@{-->}[r]& \PP^2,
	}
	\]
	where $\PP^3 \rmap \PP^2$ denotes the projection from the point $p$.
	Since $W \to \PP^2$ has connected fibres, it coincides with its Stein factorization.
	Thus, for any $n \geq 1$, the image of $W$ under the morphism given by $|-nK_W|$ is isomorphic to $\PP^2$, showing that $|-K_W|$ is not big.
	
\end{proof}

\begin{corollary}\label{relations}
	Using the notations of Proposition \ref{generalities on the involutions} and Remark \ref{birational self-map}, there exists no rank 3 fibration dominating the rank 2 fibration $X \map Y \map \Spec(\CC)$. Consequently, there are no non-trivial relations in $\BirMori(Y)$  involving  $\chi_C$.
\end{corollary}

\begin{proof}
	Let $W' \map B$ be a rank $3$ fibration dominating $X \map \Spec(\CC)$. Then $B = \Spec(\CC)$ and we have a diagram of the form
	\[
	\xymatrix@C=0.3cm@R=0.3cm{
	W' \ar@{-->}[rrd]_f \ar[rrrr] &&&& \Spec(\CC)\\
		&& X \ar[urr]
	}
	\]
	By the definition of a rank $3$ fibration $W'$ is a Mori Dream Space. Let $a$ be an ample divisor on $X$. Then there exists a composition of log-flips $g\colon W \psmap W'$ so that $g_*f^*(A)$ is nef on $W'$. With $W'$ being a Mori Dream Space itself, $g_*f^*(A)$ is semi-ample and the associated contraction gives rise to the diagram
	\[
	\xymatrix@R=0.5cm@C=0.5cm{
	W' \ar@{-->}[rd]_f \ar@{..>}[r]^g& W \ar[d]\\
			& X.
	}
	\]
	By property $(3)$ of Definition \ref{rank r fibration} $W$ is also terminal. Thus by Proposition \ref{bigness}, $-K_W$ is not big. However this would also imply that $-K_{W'}$ is not big which contradicts property $(1)$ of Definition \ref{rank r fibration}.
	The second claim follows directly from Theorem \ref{elementary relations are generators}. 
\end{proof}

\begin{remark}\label{trivial relations}
	The trivial relations involving $\chi_C$ are:
	\begin{gather*}
		(\chi_C)^2 = \id \quad \text{ and } \quad a \circ \chi_C \circ b \circ \psi^{-1}  = \id
	\end{gather*}
	where $a,b^{-1}$ are any Mori fiber space isomorphisms starting from $Y$ and $\psi$ is the Sarkisov link given by the composition $a \circ \chi_C \circ b$.
	Moreover, in the second type of relation, if $a, b \in \Aut(Y)$ with $b = a^{-1}$, then $a\circ \chi_C \circ a^{-1} = \chi_{a(C)}$.
\end{remark}

\section{Consequences}

In what follows we will stick to the notation introduced in section $3$: 
$Y$ will denote either $\PP^3$ or a smooth cubic $3$-fold in $\PP^4$ and $\mathcal{S}_{g,d}^Y$ will denote the irreducible component of the  Hilbert scheme of subvarieties of $Y$ with Hilbert polynomial $P(n) = dn - g +1$, defined in Proposition \ref{generalities on the involutions}.

\subsection{Homomorphism and semi-direct product structure.}\label{section:homoAndSDProd}

We now construct a group homomorphism from $\Bir(Y )$ to a free product $\freeprod_I \ZZ/2\ZZ$, where the indexing set I is uncountable. To do so, we will first construct a groupoid homomorphism from $\BirMori(Y )$ to the same target and then restrict it to $\Bir(Y )$.

Let $(g,d)$ be one of the pairs of Proposition \ref{generalities on the involutions}. We define the set $I_{g,d}$ to be the set of elements $\mathcal{S}^Y_{g,d}$ up to automorphisms of $Y$ and $I$ to be the disjoint union of all $I_{g,d}$ for all pairs $(g,d)$ considered in Proposition \ref{generalities on the involutions}.
The following lemma shows that $I_{g,d}$ and thus $I$ is uncountable.

\begin{lemma}\label{I is uncountable}
	For all pairs $(g,d)$ and $C \in \mathcal{S}^Y_{g,d}$ satisfying the generality conditions of Proposition \ref{generalities on the involutions}, 
	\[
	 -K_Y\cdot C \leq \dim \mathcal{S}^Y_{g,d} \leq -K_Y\cdot C  +1.
	\]
	In particular, $\dim \left(\mathcal{S}^Y_{g,d}\right) > \dim \left(\Aut(Y)\right)$. 
\end{lemma}

\begin{proof}
	By \cite[Proposition 2.8]{WeakFanos} and \cite[Proposition 3.7]{WeakFanosCubics} a general anti-canonical section $S$ containing $C$ is a smooth K3 surface (see \cite[Proposition 2.8]{WeakFanos} and \cite[Proposition 2.9]{WeakFanosCubics}). The normal bundle sequence for the embeddings $C \subset S \subset Y$ gives
	\[
	\xymatrix@C=0.5cm{
	0 \ar[r] & N_{C/S} \ar[r] & N_{C/Y} \ar[r] & N_{S/Y}|_C \ar[r] & 0.
	}
	\]
	The long exact sequence and the fact that $(C^2)_S = 2g-2$ yield
	\[
	\xymatrix@C=0.5cm@R=0.01cm{
	0 \ar[r] & H^0\left(C,\OO_C(2g-2) \right) \ar[r] & H^0(C,N_{C/Y}) \ar[r] & H^0\left(C,\OO_C(-K_X\cdot C) \right) \\
	\ar[r] & H^1\left(C,\OO_C(2g-2) \right) \ar[r] & H^1(C,N_{C/Y}) \ar[r] & H^1\left(C,\OO_C(-K_X\cdot C) \right) \ar[r] & 0.
	}
	\]
	By Serre duality, $h^1\left(C,\OO_C(-K_X\cdot C) \right) = h^0\left(C,\OO_C(2g-2+K_X\cdot C) \right)$. 
	For the six cases of $(g,d)$ and $Y$ of Proposition \ref{generalities on the involutions}, we get the following values for $2g-2+K_X\cdot C$: $-30,-26,-22,-18$ and $-12,-10$, thus $h^1\left(C,\OO_C(-K_X\cdot C) \right) = 0$.
	Similarly $h^1\left(C,\OO_C(2g-2) \right) = h^0\left(C,\OO_C \right) = 1$, thus $h^1(C,N_{C/Y})$ is either $0$ or $1$. Moreover, using the additivity of the Euler characteristic on short exact sequences and the Riemann-Roch theorem to compute we get
	\[
	h^0(C,N_{C/Y}) - h^1(C,N_{C/Y})  = -K_X\cdot C \implies -K_X\cdot C \leq h^0(C,N_{C/Y}) \leq -K_X\cdot C + 1.
	\]
	Since $C$ represents a general and thus smooth point of $\mathcal{S}^Y_{g,d}$ we get that $\dim \mathcal{S}^Y_{g,d} =h^0(C,N_{C/Y})$.
	
	For the last assertion, if $Y$ is a cubic $3$-fold, then $\dim\Aut(Y)= 0$ (see \cite{AutosOfHypersurfaces})  and we are automatically done. If $Y = \PP^3$, then in all cases $d\geq 8$ and so $-K_Y\cdot C = 4d > 15 = \dim\Aut(\PP^3)$.
\end{proof}

\begin{theorem}
	There exists a surjective group homomorphism $\psi \colon \Bir(Y) \to \freeprod_I \ZZ/2\ZZ$, which admits a section, giving the group $\Bir(Y)$ a semidirect product structure 
	\[
	\Bir(Y ) = N \rtimes \freeprod_I \ZZ/2\ZZ,
	\]
	where $N$ is the kernel of $\psi$.
\end{theorem}

\begin{proof}
	We will first define a groupoid homomorphism $\Psi\colon \BirMori(Y) \to \freeprod_I \ZZ/2\ZZ$. To do so, for each $i$ in $I$, we fix an element $C_i \in \mathcal{S}_{g,d}^Y$ in the equivalence class corresponding to $i \in I_{g,d}$. 
	The groupoid $\BirMori(Y)$ is generated by Sarkisov links and isomorphisms of Mori fiber spaces, and relations are generated by trivial and elementary ones (see Theorem \ref{elementary relations are generators}). Thus to define a groupoid homomorphism from $\BirMori(Y)$ it is enough to define it on the generators and check that all relators are mapped to the neutral element.
	With that in mind we define $\Psi$ as follows: 
	on the level of objects, $\Psi$ maps everything to the unique object of $\freeprod_I \ZZ/2\ZZ$ (when considered as a groupoid). 
	On the level of Sarkisov links and automorphisms, for each $i \in I$, $\Psi$ maps all links equivalent to $\chi_{C_i}$ (see Remark \ref{equivalent links}) to the non-zero element of the factor $i$.
	All other links and isomorphisms are mapped to the neutral element.
	Any relator not involving any link equivalent to $\chi_{C_i}$ is automatically sent to the neutral element, and the same is true for both relators of Remark \ref{trivial relations}.
	
	Define $\psi\colon \Bir(Y) \map \freeprod_I \ZZ/2\ZZ$ to be the restriction of $\Psi$ to the subgroup $\Bir(Y)$ of $\BirMori(Y)$. Since $\psi$ is the restriction of a groupoid homomorphism, it is a group homomorphism itself.
	Let $1_k$ be the non-zero element of the $k$-th factor of $\freeprod_I \ZZ/2\ZZ$, $k\in I$. Then $\psi(C_k) = 1_k$, thus the homomorphism is surjective. Conversely, we may define a section by sending $1_k$ to $\chi_{C_k}$.

\end{proof}

\begin{remark}
Using Proposition \ref{degree}, the degree of an involution $\chi_{C_i}$ and thus of an element not in the kernel of $\psi$, can be as low as $19$ in the case $Y = \PP^3$ and $11$ in the case $Y$ is a cubic $3$-fold.
\end{remark}

\subsection{Free product structure}

We now show that $\Bir(Y)$ admits a free product structure $G \ast \left(\freeprod_J \ZZ/2\ZZ\right)$.
The indexing set $J$ is defined similarly to the indexing set $I$ of the previous section: we first define $J_{g,d}^Y$ to be the set of elements of $\mathcal{S}_{g,d}^Y$ that are fixed by no non-trivial automorphism of $Y$, up to projective equivalence; then we define $J^Y$ to be the disjoint union over all pairs $(g,d)$ of Proposition \ref{generalities on the involutions} corresponding to $Y$.

A priori, it is not clear that $J^Y$ is uncountable or even non-empty. Thus we first set out to prove that $J^Y$ is uncountable. First we treat the case $Y = \PP^3$.

\begin{lemma}\label{non complete implies no autos}
	Let $C$ be a curve of genus $g \geq 2$, and let $D$ be a very ample divisor on $C$, such that $\dim \lvert D \rvert \geq 5$.

	Then for $n \geq 3$, a general $(n+1)$-dimensional subsystem $V$ of $\lvert D \rvert$ defines an embedding of $C$ in $\PP^n$ that admits no projective automorphisms. Moreover, for every such $V$, there are only finitely many other subsystems of the same dimension which are projectively equivalent. 
\end{lemma}

\begin{proof}
	The complete linear system $\lvert D \rvert$ defines an embedding to $\PP^N$ for some $N \geq 4$. Maps given by $(n+1)$-dimensional subsystems correspond to composition of the embedding with projections from $\PP^N$ to $n$-dimensional linear subspaces. Thus since $n \geq 3$, a general $(n+1)$-dimensional subsystem defines an embedding (see \cite[Propositions 3.4 and 3.5]{Hartshorne}).
	
	Recall that since the genus of $C$ is greater than or equal to $2$, by a classical theorem of Hurwitz (see \cite{MR1510753}) its automorphism group $\Aut(C)$ is a finite group. Denote by $G$ the subgroup $\{g\in \Aut(C) \, | \, g^*D \sim D\}$ of $\Aut(C)$.
	Let $V$ be an $n$-dimensional subspace of $\lvert D \rvert$. Then the automorphisms of $\PP^n$ acting on $C$ are exactly the elements of $G$ that leave $V$ invariant. If $G$ is trivial, we are done. If $G$ is non-trivial, then by considering the non-trivial action of $G$ on the Grassmanian $Gr(n+1,\lvert D \rvert)$, we see that being invariant under $G$ is a closed condition.	

	Finally, two embeddings corresponding to two subspaces $V_1$ and $V_2$ are projectively equivalent if and only if there exists $g \in G$ such that $g(V_1) = V_2$. Since $G$ is a finite group, so is the orbit of every element in $Gr(n+1,\lvert D \rvert)$, proving the second claim.
\end{proof}

\begin{lemma}\label{no projective autos}
	For $(g,d) \in \{(2,8),(6,9)\}$ and any curve $C$ of genus $g$, 	
	there exists uncountably many non-projectively equivalent curves in $\mathcal{S}_{g,d}^{\PP^3}$, isomorphic to $C$, that admit no non-trivial projective automorphisms.
	
	Consequently, $J^{\PP^3}_{2,8}$, $J^{\PP^3}_{6,9}$ and thus $J^{\PP^3}$ are uncountable.
\end{lemma}

\begin{proof}
	We will do this case by case. For $(g,d) = (2,8)$, let $D$ be a divisor of degree $8$. Since $8 \geq 4 = 2g$, $D$ is very ample and non-special and by Riemann-Roch, $\dim \lvert D \rvert = 7$. By Lemma \ref{non complete implies no autos}, a general $4$-dimensional subspace of $|D|$ defines an embedding in $\PP^3$ such that $C$ admits no non-trivial  projective automorphism. A general choice of two such subspaces gives non-projectively equivalent embeddings.
	
	For $(g,d) = (6,9)$, we start with an abstract curve of genus $6$, choose a point $p$ and define the divisor $D = K_C - p$, which is of degree $9$. By the Riemann-Roch theorem we have
	\[
	h^0(C, \OO_C(D)) = 9 - 6 +1 + h^0(C,\OO_C(K_C - D)) = 4 + h^0(C, \OO_C(p)) = 5.
	\]
	We will now show that $D$ is very ample which is equivalent to showing that for any two points $r,s$ on $C$, $h^0(C,\OO_C(D-r-s)) = h^0(C,\OO_C(D)) - 2 = 3$. 
	Suppose for contradiction that $h^0(C,\OO_C(D-r-s))= h^0(C,\OO_C(K_C - p -r-s))  = 4$. 
	We consider the canonical embedding $C_{\kappa} \subset \PP^5$ of $C$. 
	The fact that $h^0(C,\OO_C(K_C - p -r-s))  = 4$ implies that the three points $p,r$ and $s$ are collinear in the canonical embedding. 
	Write $ x_0,\dots,x_5$ for the generators of $H^0(C,\OO_C(K_C))$. Then we may compute that $h^0(C,\OO_C(2K_C)) = 15$, while $S^2\left(H^0(C,\OO_C(K_C))\right) = 21$. This implies that there are at least $6$ relations among $x_0,\dots,x_5$ in degree $2$. 
	Thus, $C$ is contained in the complete intersection of $4$ quadrics and by B\'ezout's theorem so is any tri-secant line. Consequently, there are finitely many tri-secant lines.
	Choosing a point $p$ which does not lie on any tri-secant line we get that for any $r,s \in C$, $h^0(C,\OO_C(D-r-s)) = 3$ and thus $D$ is very ample. 
	Finally, we may apply Lemma \ref{non complete implies no autos} to the divisor $D$ to obtain the desired result.
		
	%
	%
\end{proof}

\begin{remark}
	For the last two cases, namely $(g,d) = (10,10)$ and $(11,14)$, the divisors that define their embedding in $\PP^3$ are special in the Brill-Noether sense.
	Thus techniques similar to those of Lemma \ref{non complete implies no autos} are difficult to apply.
	
	However we expect the statement of Lemma \ref{no projective autos} to be true for them as well: 
	using a degeneration argument this amounts to just exhibiting one such curve with no projective automorphisms.
\end{remark}

We now treat the case $Y$ is a smooth cubic $3$-fold.

\begin{lemma}\label{no autos on cubic}
	Let $Y$ be a smooth cubic $3$-fold.
	For $(g,d) \in \{(0,5),(2,6)\}$, a general element $C \in \mathcal{S}_{g,d}^Y$ is fixed by no non-trivial automorphism of $Y$ and thus $J^Y$ is uncountable.
\end{lemma}

\begin{proof}
	Define $G \defeq \Aut(Y)\setminus \{\id\}$ and consider the correspondence
	\[
	\mathcal{F} = \left\{(C,a) \in \mathcal{S}_{g,d}^Y \times G \,|\, a(C)=C \right\}
	\] 
	together with the projections $p_1$ and $p_2$ to the two factors.
	Notice that the subset of $\mathcal{S}_{g,d}^Y$ of curves which are fixed by some automorphism coincides with the subscheme
	\[
	F \defeq \underset{a\in G}{\bigcup} p_1\left( p_2^{-1}(a)\right) \subset \mathcal{S}_{g,d}^Y.
	\]
	By \cite{AutosOfHypersurfaces} $\Aut(Y)$ is finite and so is $G$. Thus $G$ and consequently $p_1$ are projective. This implies that $F$ is closed as the finite union of the closed subschemes  $p_1\left(p_2^{-1}(a)\right)$, $a \in G$.
	We will now show that $F \neq \mathcal{S}_{g,d}^Y$, more precisely, we will show that for any $a \in G$, there exists a $C \in \mathcal{S}_{g,d}^Y$ not fixed by $a$.

	We briefly recall a construction from \cite[Section 3.3]{WeakFanosCubics}: 
	let $p$ be a general point in $Y$ and $S$ a general hyperplane section of $Y$ not containing $p$. 
	Define the rational map $\phi \colon S \rmap Y$ by sending a point $q$ to the third point of intersection of the line through $p$ and $q$ and $Y$. 
	Then $Q \defeq\overline{\phi(S)}$ is a hyperquadric section of $Y$ singular at the point $p$. 
	Moreover, $Q$ is isomorphic to the blowup of $\PP^2$ along 12 points, all lying on a cubic curve $\Gamma$, followed by the contraction of $\Gamma$. 
	Using this construction, the authors provide examples of curves of genus $g$ and degree $d$ with $(g,d) \in \{(0,5),(2,6)\}$ lying on $Q$ and passing though $p$, satisfying the generality conditions of Proposition \ref{generalities on the involutions}.

	Now let $p$ be a general point in $Y$ such that $a(p) = q \neq p$. Choose a general hyperplane section $S$ of $Y$ not containing $p$, such that the hyperquadric section $Q$ of the previous construction does not contain $q$. Then for $(g,d) \in \left\{ (0,5),(2,6) \right\}$, we may find a curve $C \in \mathcal{S}_{g,d}^Y$ lying on $Q$ and passing though $p$. We have $p \in C$ but $a(p) \notin C$, thus $a(C) \neq C$.
\end{proof}

\begin{theorem}
	For each $j$ in $J$, we fix an element $C_j \in \mathcal{S}_{g,d}^Y$ in the projective equivalence class corresponding to $j \in I_{g,d}$. 
		
	Let $G$ be the subgroup of $\Bir(Y)$ generated by elements admitting a decomposition into Sarkisov links none of them equivalent to $\chi_{C_j}$ (see Remark \ref{equivalent links}).
	We then have
	\[
	\Bir(Y) = G \ast \left(\freeprod_{J^Y} \langle \chi_{C_j} \rangle \right) \isom G \ast \left(\freeprod_{J^Y} \ZZ/2\ZZ\right),
	\]
	where the indexing set $J$ is uncountable.
\end{theorem}

\begin{proof}
	The groupoid $\BirMori(Y)$ is generated by Sarkisov links and isomorphisms of Mori fiber spaces, and relations are generated by trivial and elementary ones (see Theorem \ref{elementary relations are generators}). 
	Every link equivalent to $\chi$ (see Remark \ref{equivalent links}) is of the form $a\circ \chi \circ b$ and is thus redundant in the generation of the groupoid. 
	Thus we may take as generators Mori fiber space isomorphisms, as well as all Sarkisov links that are either $\chi_{C_j}$ or they are not equivalent to $\chi_{C_j}$ for any $j \in J$.
	Moreover, by replacing links equivalent to $\chi_{C_j}$ by $a\circ \chi_{C_j} \circ b$ in all relations, for any $j \in J$, we see that the only generating relations involving $\chi_{C_j}$ of Remark \ref{trivial relations} are $\chi_{C_j}^2=\id_Y$ and $a\circ \chi_{C_j}\circ b=\chi_{C_j}$. In the second relation, the target and the source of $\chi_{C_j}$ being $Y$, implies that $a,b \in \Aut(Y)$.
	However, by comparing base loci, we see that $a$ and $b$ must fix the curve $C_j$, which by our choice of $J^Y$, implies that $a = d = \id_Y$. 
	Thus the only relation among our new set of generators, involving $\chi_{C_j}$ is $\chi_{C_j}^2=\id_Y$.
	
	To show that $\Bir(Y) = G \ast \left(\freeprod_{J^Y} \langle \chi_{C_j} \rangle \right)$, we have to show that:
	\begin{enumerate}
		\item each element of $\Bir(Y)$ can be written as a product of elements in the factors of $G \ast \left(\freeprod_{J^Y} \langle \chi_{C_j} \rangle \right)$;
		\item generating relations involve only elements from a single factor of $G \ast \left(\freeprod_{J^Y} \langle \chi_{C_j} \rangle \right)$.
	\end{enumerate}
	For the former, given any element of $\Bir(Y)$ we may decompose it into Sarkisov links using the generators chosen in the previous paragraph. 
	Then factoring this decomposition by isolating all elements $\chi_{C_j}$, we get a product of elements in $G$ and $\langle \chi_{C_j} \rangle$, $j\in J$.

	As for the latter, let $r = \id_Y$ be a relator in $\Bir(Y)$. 
	As previously, $r$ is a product of conjugates of the generating relations chosen in the first paragraph, these are precisely elements of the form ${\chi_{C_j}}^2$, $j\in J$ and $R$, with $R = \id_W$ is a relator in $\BirMori(Y)$ involving none of the $\chi_{C_j}$.
	Again factoring by isolating all expressions ${\chi_{C_j}}^2$ we get that $r$ is a product of conjugates of elements of the form $r_G$ and ${\chi_{C_j}}^2$, where $r_G = \id_Y$ is a relation in $G$. 
	Thus $r$ may be generated by relators involving only elements of $G$ or $\langle \chi_{C_j} \rangle$, $j\in J$.

	For the last assertion, if $Y = \PP^3$ we conclude by Corollary \ref{no projective autos} and otherwise by Lemma \ref{no autos on cubic}.
\end{proof}

\begin{remark}\label{many isos}
	The construction of the isomorphism above depends on the choice of a curve in each projective equivalence class of $\mathcal{S}_{g,d}^Y$. Different choices give rise to different isomorphisms.
\end{remark}

\subsection{Inner and Field Automorphisms}\label{section:InnerAndFieldAutos}

We now construct a group automorphism of $\Bir(Y)$ which we show that is not a product of inner and field automorphisms (see Subsection \ref{subsection:intro nongeneration} for all the relevant definitions).

We first fix an isomorphism $\Bir(Y) \isom G \ast \left(\freeprod \ZZ/2\ZZ\right)$ among the ones constructed in the previous section (see Remark \ref{many isos}). 
Choose a non-trivial permutation $\rho$ of $J$, such that there exists $j_0 \in J_{g,d}^Y$ with ${j_0}' \defeq \rho(j_0) \in J_{g',d'}^Y$ and $(g,d) \neq (g',d')$.
We note that, whether $Y$ is $\PP^3$ or a smooth cubic $3$-fold, such a choice is always possible, for example, for $Y = \PP^3$, we can choose $(g,d)= (2,8)$ and $(g',d')=(6,9)$ since the corresponding $J_{g,d}^Y$ and $J_{g',d'}^Y$ are non-empty by Lemma \ref{no projective autos}.

We now define an automorphism $\phi = \phi(\rho)$ of $\Bir(Y)$ by sending the factor of the free product with index $j$ to that with index $\rho(j)$.
More precisely, we define the automorphism $\phi = \phi(\rho)$ on the generators of the free product by sending  $\chi_{C_j}$ with $\chi_{C_{\rho(j)}}$ and fixing all generators in $G$.

\begin{proposition}
	The automorphism $\phi$ of $\Bir(Y)$ defined above is not the composition of a  
	field automorphism $\sigma$ of $\CC$ preserving $Y$ and an inner automorphism of $\Bir(Y)$.
\end{proposition}

\begin{proof}
	Suppose the contrary. Then for any $f \in \Bir(Y)$ we have
	\[
	\phi(f) = b_{\tau} \left( g \circ f \circ g^{-1} \right),
	\]
	where $g \in \Bir(Y)$ and $\tau$ is a field automorphism of $\CC$ (see Subsection \ref{subsection:intro nongeneration} for the definition of $b_{\tau}$).
	For $f = \chi_{C_{j_0}}$ we get
	\[
	\chi_{C_{{j_0}'}} 
	= b_{\tau} \left( g \circ \chi_{C_{j_0}} \circ g^{-1} \right) 
	\iff b_{\sigma}(\chi_{C_{{j_0}'}}) = g \circ \chi_{C_{j_0}} \circ g^{-1},
	\]	
	where $\sigma = \tau^{-1}$. 
	However, by the description of relations involving $\chi_{C_{j_0}}$, the only possible choice would be for $g = \id_Y$.
	We would then have
	\[
	b_{\sigma}(\chi_{C_{{j_0}'}}) = \chi_{C_{j_0}}.
	\]
	By comparing base loci, we get that $a_{\sigma}(C_{{j_0}'}) = C_{j_0}$. However, $a_{\sigma}(C_{{j_0}'})$ is abstractly isomorphic to $C_{{j_0}'}$ which cannot be isomorphic to $C_{j_0}$ as they have different genera, which is a contradiction.
\end{proof}

\begin{corollary}
	The automorphism group of $\Bir(Y)$ is not generated by inner automorphisms and field automorphisms preserving $Y$.
	
	Moreover, there exist elements of any order which do not lie in the subgroup generated by inner and field automorphisms: the order of $\phi(\rho)$ is equal to the order of $\rho$ and since $J$ is infinite we can find permutations of any order.
\end{corollary}

\begin{remark}
	For $Y = \PP^3$ and any $\rho$ as above, the group automorphism $\phi(\rho) \colon \Cr_3(\CC) \to \Cr_3(\CC)$ is not a homeomorphism with respect to either the Zariski or the Euclidian topology on $\Cr_3(\CC)$. 
	Indeed by the results of \cite{ContinuousAutomorphisms}, any homeomorphism of $\Cr_3(\CC)$, with respect to either of the two topologies, is the composition of a field automorphism with an inner automorphism.
\end{remark}

\subsection{Extensions of the construction}
All results proven in the previous sections rely on the involutions constructed in Proposition \ref{generalities on the involutions} and their rigidity proved in Proposition \ref{bigness}. These involutions have appeared before in the literature in \cite{Cutrone-Marshburn}, \cite{WeakFanos} and \cite{WeakFanosCubics}.

The rigidity of these involutions essentially boils down to the fact that they are dominated by a smooth weak-Fano $3$-fold of anti-canonical degree $2$, which is the smallest degree possible. However, among the lists of \cite{Cutrone-Marshburn} there are several other examples of involutions of Fano $3$-folds which have the same property. Thus it is a natural question whether the whole construction extends to these cases as well.

\bibliography{rigidInvolutions}{}
\bibliographystyle{alpha}

\end{document}